\documentclass[12pt,letterpaper]{amsart}
\usepackage{hyperref}

\input{mathdefs.sty}

\renewcommand{\myboldfont}{\mathbb}

\usepackage[left=2.5cm, right=2.5cm , top=3.2cm, bottom=3.2cm]{geometry}

\usepackage{bm}
\usepackage{enumerate}
\usepackage{bbm}

\newtheorem{theorem}{Theorem}
\newtheorem{lemma}[theorem]{Lemma}

\newtheorem{proposition}[theorem]{Proposition}

\newcounter{theremark}
\newcommand{\1}{\mathbbm{1}}

\numberwithin{equation}{section}

\usepackage{graphicx}

\newcommand{\sumstar}{\sideset{}{^\star} \sum}

\usepackage[font=small,labelfont=bf]{caption}

\makeatletter

\renewcommand{\section}{\@startsection{section}{0}{0pt}{3.6ex plus 0.2ex minus 0.1ex}{2.3ex plus 0.1ex minus 0.1ex}{\center\normalfont\sc\large}}

\def\imod#1{\allowbreak\mkern5mu{\operator@font mod}\,\,#1}

\makeatother

\newcommand{\NN}{\mathbb{N}}
\newcommand{\ZZ}{\mathbb{Z}}

\newcommand{\smallabcd}{\ensuremath{\left(\begin{smallmatrix}a & b\\c& d\end{smallmatrix}\right)}}

\begin{document}

\title{Effective equidistribution of horocycle lifts }
\author{Ilya Vinogradov}
\address{Department of Mathematics, Princeton University, Princeton, NJ 08544, United States} 
\email{ivinogra@math.princeton.edu}
\date{\today}

\begin{abstract}
We give a rate of equidistribution of lifts of horocycles from the space $\SL(2,\Z)\quot \SL(2, \R)$ to the space $\ASL(2,\Z)\quot \ASL(2,\R)$, making effective a theorem of Elkies and McMullen.
This result constitutes an effective version of Ratner's measure classification theorem  for measures supported on general horocycle lifts.
The method used relies on Weil's resolution of  the Riemann hypothesis for function fields in one variable and generalizes the approach of Str\"ombergsson to the case of linear lifts and that of Browning and the author to rational quadratic lifts. 
\end{abstract}


\maketitle

\thispagestyle{empty}


\section{Introduction}

\subsection{Background}


In the theory of flows on homogeneous spaces, Ratner's theorems on measure rigidity, topological rigidity, and orbit equidistribution \cite{ratner_raghunathans_1991, ratner_raghunathans_1991_1} play a major role. Their applications go far beyond the realm of dynamical systems and include results in number theory and mathematical physics \cite{ElkiesMcM04, shah_equidistribution_2009, marklof_frobenius_2010,  marklof_strombergsson_free_path_length_2010}; thorough expositions and comprehensive references 
may be found in \cite{morris_ratner_2005}.

In the last decade there has been an increased interest in obtaining \emph{effective} versions of
Ratner's results, such as giving a rate of convergence of measures in the measure rigidity theorem.  There are
two general situations where effective results may be
proved: when the group generating the flow is horospherical, or when it is  ``large'' in an
appropriate sense (cf.~\cite[Sec.~1.5.2]{einsiedler_effective_2009}). 
Recently, rates of convergence were obtained for several settings where the corresponding group is neither horospherical nor large. Green and Tao \cite{green_tao_quantative_2012} proved effective
equidistribution of polynomial orbits on nilmanifolds, while Einsiedler, Margulis, and Venkatesh \cite{einsiedler_effective_2009}
proved effective equidistribution for closed orbits of semisimple groups on
general homogeneous
spaces. Str\"ombergsson \cite{strombergsson_effective_2013} and Browning and the author \cite{browning_vinogradov_2016} gave rates for the convergence of measures on special horocycle lifts; the present paper further explores this direction by giving a rate of convergence for measures on general horocycle lifts.

%


\subsection{Results}
For $x\in \R$ and $y>0$, let
\begin{align} 
n(x)&=\begin{pmatrix}1 & x\\0 & 1\end{pmatrix},&  a(y)&=\begin{pmatrix}\sqrt y& 0\\ 0 & 1/\sqrt y\end{pmatrix}.
\end{align}
It is a fundamental result in homogeneous dynamics that long closed horocycles $\{n(x)\colon \strut x \in [-\frac12,\frac12)\}$ on $X=\sltz\quot \sltr$ equidistribute under the geodesic flow $a(y)$ as $y\to 0$. That is, for every bounded continuous $f\colon X\to \R$, 
\beq
\lim_{y\to 0}\int_{-\frac 12}^{\frac 12} f(n(x)a(y))\, dx = \int_X f(g) \, d\mu_X(g),
\eeq 
where $\mu_X$ is the Haar probability measure on $X$. This can be proved using thickening followed by applying the mixing property of $a(y)$ \cite{margulis_some_aspects_2004}, which is a general approach when the integral is taken over all unstable directions of a flow. The rate of convergence was given in \cite{zagier_eisenstein_1979, sarnak_asymptotic_1981} and is related to the zero-free region for the Riemann zeta function. It is proved that for $f\in C^{\infty}_0 (X)$, 
\beq
\int_{-\frac 12}^{\frac 12} f(n(x)a(y))\, dx = \int_X f(g) \, d\mu_X(g) + o_f(y^{1/2}),
\eeq 
where the error term depends on the error term in the Prime Number Theorem. 
In the present paper we establish a similar result for certain horocycle lifts. 

Let $G=\ASL(2,\R)=\sltr\ltimes \R^2$, and set $\Gamma=\ASL(2,\Z)=\SL(2,\Z)\ltimes \Z^2$, which is a lattice in $G$. We view elements of $G$ as ordered pairs $(M,\bm x)$ with $M\in \sltr$ and $\bm x\in \R^2$, and multiply them following the rule
\beq
(M,\bm x)(M',\bm x')=(MM',\bm x M' + \bm x')
\eeq
thinking of $\bm x, \bm x'$ as row vectors in $\R^2$. Writing $Y=\Gamma\quot G$, we equip this homogeneous space with the Haar probability measure that we denote $\mu_Y$. When no confusion can arise we shorten $(M,\bm 0)$ to $M$. 

It is important to note the relationship between $X$ and $Y$, the latter being a bundle over the former with two-dimensional torus fiber. The space $X$ parametrizes unimodular lattices in $\R^2$, while $Y$ is the space of lattice translates in $\R^2$. Thus, each point in $X$ corresponds to a lattice $\Lambda\subset \R^2$, and a choice of $\bm x\in \Lambda\quot \R^2$  determines the translated lattice $\Lambda + \bm x\subset \R^2$, which corresponds to a point of $Y$. 

For a continuous function $\bm \xi =(\xi_1,\xi_2)\colon \R \to \R^2$ define 
\beq
\tilde n(x) = (\1,\bm \xi(x)) n(x);
\eeq
we call this is a \emph{lift} of a horocycle from $X$ to $Y$. Let $\rho\colon \R\to \R$ be nonnegative, continuously differentiable, supported on a compact interval (without loss of generality,  $\supp\rho\subset (-1,0]$), and of integral $1$. It is natural to ask whether the lifted measures $\nu_y$ defined by 
\beq
\int_{\R} f(\tilde n(x) a(y))\, \rho(x) dx = \int_Y f(g) \, d\nu_y(g) \label{eq:measures}
\eeq
have a weak-* limit as $y\to0$.
Using Ratner's Theorem \cite{ratner_raghunathans_1991}, Elkies and McMullen \cite{ElkiesMcM04} established a condition on $\vecxi$ under which $\nu_y$ converges to $\mu_Y$. Let $\Xi(x)=x\xi_1(x)+\xi_2(x)$. A horocycle lift is called \emph{rationally linear} if for some $\alpha,\beta\in\Q$, 
\beq
\leb\{x\colon \Xi(x)=\alpha x+\beta\}>0.\label{eq:rational}
\eeq

\begin{theorem}[{\cite[Th.~2.2]{ElkiesMcM04}}] Suppose that a horocycle lift $\tilde n$ is \emph{not} rationally linear, that $\Xi$ is Lipschitz, and that $\xi_1$ is continuous. Then, $\nu_y\to \mu_Y$ in the weak-* topology as $y\to 0$. 
\end{theorem}

In the present paper, we make convergence in this theorem effective, which requires an effective version of rational nonlinearity. We say $\Xi$ is \emph{$D$-nice} for some $D\ge 2$ if $\Xi$ is twice continuously differentiable and there exist $x_0 \in\R$ and $C_1, C_2>0$ such that 
\begin{align}
 C_1|x-x_0|^{D-2} \le |\Xi''(x)| \le C_2|x - x_0|^{D-2}
\end{align}
for every $x$ in the support of $\rho$. For such a lift, set $C = \max\left\{C_1^{-1/2}, C_2^{1/2}\right\}.$ We say that $\tilde n$ is \emph{$D$-nice}  if the corresponding $\Xi$ is $D$-nice.


\begin{theorem}\label{th:main}
Fix a density function $\rho$ as before, and let $\tilde n$ be $D$-nice. 
Assume that $f$, $\rho$, and $\tilde n$ are such that all norms in \eqref{eq:constant} are finite.  Then for every $\eps>0$ there exists a constant $C(\eps, f,\rho,\tilde n)$ such that 
\beq
\left\lvert \nu_y(f) - \mu_Y(f) \right\rvert \le   
 C(\eps,f,\rho,\tilde n)\, y^{\min\left\{\frac1{16},\frac1{2D}\right\}-\eps}
\eeq
for all $y\in(0,1).$ Moreover, we can take
\begin{align}\label{eq:constant}
 C(\eps,f,\rho,\tilde n) = A(\eps,\eta) C \|f\|_{C^8_{\text b}}  
 (C + \|\Xi\|_{C^1_{\text b}} + \|\xi_1\|_{L^\infty})
 \|\rho\|_{W^{1,1}}^{1-\eta} \|\rho\|_{W^{2,1}}^\eta 
\end{align}
for any $\eta\in(0,1)$, and the function $A$ is universal. 
\end{theorem}

We observe that the $\eps$-loss in the error term can be replaced by a logarithmic loss with a slightly more tedious computation as in  \cite{browning_vinogradov_2016}. The overall error would then be a constant times $y^{\min\left\{\frac1{16},\frac1{2D}\right\}} \log^{\kappa}(2+1/y)$ for some $\kappa>0$. The norms used to define $C(\eps, f,\rho, \tilde n)$ are rigorously defined in \eqref{eq:norm}, \eqref{eq:sobolev_norm}.

\subsection{Discussion}
This paper stands in the series of works that prove effective equidistribution results in the setting of a sequence on measures supported on the unstable manifold of a diagonal flow. The first and by now classical is \cite{sarnak_asymptotic_1981}; it gives the optimal rate of equidistribution of long closed horocycles on quotients of the $\SL(2,\R)$ by using Eisenstein series to relate this question to the zero-free region of the Riemann zeta function. The case of non-uniform measure on the horocycle was treated by Str\"ombergsson \cite{strombergsson_uniform_2004}; this work also proves effective equidistribution for horocycle pieces of optimal intermediate length (length of piece can be nearly as short as the square root of the length of the horocycle). 
Horocycle lifts to $Y$ were first studied in \cite{ElkiesMcM04} where an ineffective equidistribution theorem for general lifts is proved. Str\"ombergsson \cite{strombergsson_effective_2013} used number-theoretic techniques similar to those employed in the present paper to give a rate for equidistribution of linear irrational lifts (in our notation, these correspond to $\vecxi$ being a constant that is not in $\Q^2$ and our results do not apply as $\Xi'' = 0$). The rate in this setup depends on the Diophantine properties of $\vecxi$. The method of Str\"ombergsson was further developed in \cite{browning_vinogradov_2016} to treat the case of rational quadratic lifts, with the case  $\vecxi(x) = (x/2, - x^2/4)$ being the most interesting. The treatment of this particular lift yields a rate for the convergence of the gap distribution of the sequence $\sqrt n \imod 1$. 

The present paper completes the effectivization of equidistribution theorems for lifts. The powers of $y$ (up to $y^{\eps}$) that appear in error terms in the aforementioned theorems are 
\begin{align}
 y^{1/2}, & \text{ Sarnak \cite{sarnak_asymptotic_1981},} \label{eq:sarnak_rate}\\
 y^{1/4}, & \text{ Str\"ombergsson \cite{strombergsson_effective_2013}, assuming best Diophantine condition,}\\
 y^{1/4}, & \text{ Browning, V. \cite{browning_vinogradov_2016},}\\
 y^{1/16}, & \text{ present work, assuming best lift}.
\end{align}
We also remark that $y^{3/4}$ in \eqref{eq:sarnak_rate} would be equivalent to the Riemann hypothesis. The novelty of the present paper is that it does not make use of quadratic niceties of \cite{browning_vinogradov_2016} or linear simplicity of \cite{strombergsson_effective_2013}, allowing for the treatment of general lifts. 
The key result is Proposition \ref{prop:cancellations}, which establishes cancellations in a certain Kloosterman-like exponential sum \eqref{eq:sum_l} for all values of the indices involved. 

In addition to the theorems mentioned above, we must also mention the recent result of Ubis \cite{ubis_effective_2016}, who used the  ``Fourier method'' on $\R^d$ to prove effective equidistribution of certain manifolds on $(\Gamma\quot \SL(2,\R))^d$. Fix $\tilde a(y) = (a(y),\dots, a(y)) \in \sltr^d$ and consider its $d$-dimensional unstable manifold. Then, given a submanifold that is ``totally curved'' and has positive codimension, Ubis gives a rate of equidistribution of this submanifold under the action of $\tilde a(y)$. Although the method of this paper is different from that of the present work, the setup is quite similar, which gives hope that other equidistribution statements of this flavor  (for example, results of Shah on equidistribution of curves \cite{shah_limiting_2009, shah_asymptotic_2009, shah_equidistribution_2009}) will be effectivized in the near future. 

Related results on effective equidistribution for $\sltr$ include papers of Tanis and Vishe \cite{tanis_uniform_2015} and Flaminio, Forni, and Tanis \cite{flaminio_effective_2015} on period integrals, both building on the work on the seminal paper of Flaminio and Forni \cite{flaminio_invariant_2003} on invariant distributions for the horocycle flow. Effective equidistribution of ``relatively large'' orbits is proven by Einsiedler, Margulis, and Venkatesh \cite{einsiedler_effective_2009}. 

Another direction of refining convergence in Ratner's theorem is extending weak-* convergence to unbounded test functions, known as the problem of \emph{convergence of moments}. 
This question was answered affirmatively in certain situations, relating to theta functions and their application to values of inhomogeneous quadratic forms \cite{marklof_pair_correlation_2003};
the pair correlation function of the sequence $\sqrt n$ modulo $1$ \cite{EMV_two_point}; 
directions of Euclidean lattice points \cite{EMV_directions_2013};
directions of hyperbolic lattice points \cite{directions_hyperbolic}. 

The question of convergence of moments is open for the main result of this paper, Theorem \ref{th:main}, as is the question of the rate of equidistribution of the unipotent flow $\{\tilde n(x) : x\in\R\}$ with $\vecxi(x) = (x/2, -x^2/4).$ We hope to return to these questions in future work. 

\subsection{Plan of paper}
Section \ref{sec:number} contains an application of number-theoretic techniques to control a special exponential sum. In Section \ref{sec:fourier}, we single out the main term from  the integral in the statement of Theorem \ref{th:main}; we then bound the error term in Section \ref{sec:error}. 

\subsection{Notation}
Given functions $f,g\colon S\rightarrow \R$, with 
$g$ positive,  we will write $f\ll g$ if 
there exists a constant $c$ such that 
$|f(s)|\leq cg(s)$ for all $s\in S$.

\subsection{Acknowledgements}
The author would like to thank Trevor Wooley for useful discussions and Tim Browning and Jens Marklof for comments on earlier versions of this paper. 
The research leading to these results has received funding from the European Research Council under the European Union's Seventh Framework Programme (FP/2007--2013) / ERC Grant Agreement n.\ 291147.

\section{Special exponential sums\label{sec:number}}
In this section we make a detailed examination of the 
exponential sum
\begin{align}
 \label{eq:sum_l}
 S_c(k,l,n) = S = \sum_{\substack{(c,d) = 1\\ 0\le d < c}} e\left( \frac{l\bar d - kd}c - nc \Xi\left( - \frac dc\right)\right).
\end{align}
We distinguish two cases that require different treatment, according as $l=0$ or $l\ne 0$. 

Consider first the case $l=0$; the cancellations in the sum $S$ must come from analytic properties of $\Xi$. Writing
\begin{align}
 S = \sum_{h\mid c} \mu(h) \sum_{d=1}^{c/h} e\left(-nc \Xi\left(-\frac{dh}c\right) - \frac{kdh}c\right) = \sum_{h\mid c} \mu(h) S\left(\frac ch\right),
\end{align}
we massage the inner sum over $d$. 
We have 
\begin{align}
 S(x) = \sum_{d=1}^x e\left( -hnx\Xi\left(-\frac dx\right) - \frac{kd}x\right) = \sum_{d=1}^x e(w(d)).
\end{align}
Since $w''(d) = - \frac{hn}{x} \Xi''\left(-\frac dx\right)$, our assumption on $\Xi$ implies that 
$|w''(d)| \asymp_{C_1,C_2} \frac{hn}x \lvert-\frac dx - z_0\rvert ^{D-2}.$ Let $\delta>0$. Using the van der Corput estimate (cf.\ \cite[p.~8]{graham_kolesnik_1991}) when $\lvert-\frac dx - z_0\rvert > \delta$ and the trivial estimate otherwise, we get the bound
\begin{align}
 |S(x)| \ll \delta x + C_2^{1/2} h^{1/2} n^{1/2} x^{1/2} \delta^{\frac{D-2}2} + \frac{x^{1/2}}{C_1^{1/2} h^{1/2}n^{1/2}} \delta^{-\frac{D-2}2}.
\end{align}
The optimal choice for $\delta$ is $x^{-1/D}$, giving the bound $C h^{1/2}n^{1/2} x^{1-1/D}$ for $S(x)$, where $C = \max\left\{C_2^{1/2}, C_1^{-1/2}\right\}.$
The contribution of the case $l=0$ is thus
\begin{align}
 S & \ll C \sum_{h\mid c} |\mu(h)| h^{1/2} n^{1/2} \left(\frac ch \right)^{1-1/D}\\
 & \ll C n^{1/2} c^{1-1/D} \sum_{h\mid c} |\mu(h)| h^{1/D+1/2-1}\\
 & \ll_\eps C n^{1/2} c^{1-1/D+\eps}.
\end{align}

Consider now the case $l\ne 0$. We adopt the convention that the range of summation includes only those values of the indices for which the summands are defined, allowing us to drop the coprimality condition. We begin by applying the Weyl-van der Corput inequality \cite[eq.~(2.3.5)]{graham_kolesnik_1991} for some $H\in[1,c]$ to be chosen later, which gives 
\begin{align}
 |S|^2 & \ll \frac{c^2} H + \frac cH \sum_{1\le h\le H} \bigg\lvert\sum_{d = 1}^c e\left(\frac{l(\overline{d+h} - \overline d)}c - nc \left( \Xi\left( - \frac{d+h}c\right) - \Xi \left( - \frac dc\right)\right)  \right)\bigg\rvert.
\end{align}
Writing 
\begin{align}
 a_d & = e\left(\frac{l(\overline{d+h} - \overline d)}c\right),
 &
 b_d &
 =e\left(- nc \left( \Xi\left( - \frac{d+h}c\right) - \Xi \left( - \frac dc\right)\right)\right), 
\end{align}
we set 
\begin{align}
 T = \sum_{d=1}^c a_d b_d;
\end{align}
here $T = T_c(h,l,n)$ depends on  $c$, $h$, $l$, and $n$; and we follow the  convention that the terms with $a_d$ undefined are assumed to be zero. Now we seek to get cancellations in the sum $T$. Summing by parts, we can write
\begin{align}
 \label{eq:T_by_parts}
 T &
 =
 b_c\sum_{q=1}^c a_d + \sum_{d=1}^{c-1} \sum_{q=1}^d a_q (b_d - b_{d+1}) ,
\end{align}
provided $c\ge 2$ (when $c=1$, the bound $T \ll 1$ is satisfactory). 
Set $A_d = \sum_{q = 1}^d a_q$ and $B_d = b_d -b_{d+1}$;
we need to bound $A_d$ and $B_d$. For the first, we use smoothing to write the sum as a complete sum modulo $c$ followed by standard estimates for exponential sums; for the second, we rely on Taylor's theorem and smoothness of $\Xi$. 

Let $\delta\in (0,1)$ be a number we will choose later depending on $c$, and let $I_\alpha \colon [0,1] \to \{0,1\}$ be the indicator of $[0,\alpha]$ for $\alpha\in [0,1]$. Let $\psi\colon \R \to\R$ be smooth, of integral $1$, supported on $[-1,1]$. Then, $\psi_\delta(x) = \frac1\delta \psi\left(\frac x \delta\right)$ is smooth, of integral $1$, supported on $[-\delta,\delta]$. Set 
\begin{align}
 I^{\delta}_{\alpha} (x) = 
 \begin{cases}
 I_{\alpha + \delta} * \psi_\delta(x),& \alpha+\delta\le 1\\ 
 I_{1}(x), & \alpha+\delta >1.
\end{cases}
\end{align}
Using the notation where $e_c(\cdot) = e(\tfrac{\cdot}c)$, we can write
\begin{align}
 A_d 
 & = \sum_{q=1}^c e_c(l(\overline{q+h} - \overline q)) I_{d/c}(q/c)\\
 & = \sum_{q=1}^c e_c(l(\overline{q+h} - \overline q)) I_{d/c}^\delta (q/c) + O(\delta c).
\end{align}
We introduce quantities 
\begin{align}
 c_{k,d}^\delta & = \int_0^1 I_{d/c}^\delta (x) e(-kx) \, dx, &
  U_c(h, k,l)& = \sum_{q=1}^c e_c(l(\overline{q+h} - \overline q) + kq).
\end{align}
Then, we can write 
\begin{align}
 A_d = \sum_{k\in\Z} c^\delta_{k,d} U_c(h, k,l) +O(\delta c).\label{eq:need_num_theory}
\end{align}

Now the sum $U_c(h,k,l)$ can be treated using results of Bombieri \cite{bombieri} for $c$ a prime, generalized by Cochrane and Zheng \cite{cochrane-zheng} for $c$ a prime power. 
We begin by recording the easy multiplicative property
\begin{equation}\label{eq:mult}
U_{q_1q_2} (h, k,l) = U_{q_1}(h, k\bar q_2, l\bar q_2, ) U_{q_2} (h, k\bar q_1, l\bar q_1)
\end{equation}
whenever  $q_1,q_2\in \N$ are coprime and $\bar q_1,\bar q_2\in \Z$ satisfy $q_1\bar q_1+q_2\bar q_2=1.$
This renders it sufficient to study $U_{p^m}(h, k,l)$ for a prime power $p^m$.
We may write $U_{p^m}(h,k,l)$
in the form
\begin{equation}\label{eq:fall}
\sumstar_{\substack{q\imod{p^m}}} e_{p^m}\left(\frac{f_1(q)}{f_2(q)}
\right),
\end{equation}
where $f_1(q) = kq^2(q+h) - hl$ and $f_2(q)=q(q+h)$.
The symbol 
 $\sum^\star$ emphasizes the fact that $q$ is only taken over values for which  $q\nmid f_2(q)$, 
 in which scenario  $f_1(q)/f_2(q)$ means 
$f_1(q)\overline{f_2(q)}$.
We proceed by establishing the following result, which is far form optimal, but sufficient for our needs. 

\begin{lemma}\label{lem:primes}
Let $p$ be a prime and $m\in \N$. Then we have 
\begin{align}
U_{p^m}(h,k,l)
\ll
\begin{cases}p^{1/2}(p,(k,hl))^{1/2} , & m=1,
\\
p^{2m/3} (p^m,(k,lh))^{m/3}, & m>1.
\end{cases}
\end{align}
\end{lemma}

\begin{proof}
When $m=1$, the result follows from \cite[eq.~(1.2)]{cochrane-zheng}, which is a restatement of Bombieri's result \cite{bombieri}. 

When $m>1$, we use  \cite[Cor.~3.2]{cochrane-zheng}. In their notation, we have $d(f_1) = 3$, $d(f_2) = 2$, $d(f) = 5$, $d^*(f) = 3$,
\begin{align}
 d_p(f) = 
 \begin{cases}
  0, & p\mid k, p\mid hl,\\
  2, & p\mid k, p\nmid hl,\\
  1, & p\nmid k, p\mid hl,\\
  5, & p\nmid k, p\nmid hl.\\
 \end{cases}
\end{align}
In the last three cases, $d^*_p(f) = 2,1,3$, respectively, so that, by  \cite[Cor.~3.2]{cochrane-zheng}, $U_{p^m}(h,k,l) \ll p^{2m/3}$, which is satisfactory. In the first case, we choose $t\in\N$ so that $p^t\| (k,hl)$. If $t\ge m$, the trivial bound on $U_{p^m}(h,k,l)$ is satisfactory. If $t < m$, we write $t = t_1 + t_2$, where $p^{t_1} \mid h$ and $p^{t_2} \mid l$,
\begin{align}
 U_{p^m}(h,k,l) = p^t U_{p^{m-t}}(hp^{-t_1}, kp^{-t}, lp^{-t_2}) \ll p^t p^{2(m-t)/3} = p^{2m/3+t/3},
\end{align}
which is satisfactory.
\end{proof}

We write $c=uv$, where $u$ is square-free and $v$ is square-full. That is, $p\mid u$ implies $p^2 \nmid u$ and $p\mid v$ implies $p^2 \mid v$. 
Using the multiplicativity property \eqref{eq:mult}, we may  
apply Lemma \ref{lem:primes} for different primes
to   arrive at the following result. 

\begin{lemma}\label{lem:cancellations}
 Let $c\in \NN$ and let $h, k, l\in \ZZ$.
Then for every $\eps>0$ we have 
\begin{align}
U_{c}(h,k,l)
\ll_\eps
c^{\eps} 
u^{1/2} (u, (k,hl))^{1/2}
v^{2/3} (v, (k,hl))^{1/3}.
\end{align}
\end{lemma}

We substitute this bound into \eqref{eq:need_num_theory} together with the bound for Fourier coefficients
\begin{align}
 c_{k,d}^\delta \ll_{\gamma,\psi} \frac1{1+k} \cdot \left(\frac{1}{k\delta + 1}\right)^\gamma
\end{align}
for $\gamma \ge 0$.  Choosing $\gamma = \eps$, we  get
\begin{align}
 A_d & \ll_{\eps} \sum_{k\in\Z} 
 c^{\eps} 
u^{1/2} (u, (k,hl))^{1/2}
v^{2/3} (v, (k,hl))^{1/3}
\frac1{1+k} \cdot \left(\frac{1}{k\delta + 1}\right)^\eps + O(\delta c)
\\
&
\ll  \sum_{k\in\Z} 
 c^{\eps} 
u^{1/2} (u, hl)^{1/2}
v^{2/3} (v, hl)^{1/3}
\frac{(c,k)^{1/2}}{1+k} \cdot \left(\frac{1}{k\delta + 1}\right)^{\eps} + O(\delta c).
\end{align}
Now we observe that 
\begin{align}\label{eq:gcd_av}
 \sum_{k=1}^K (c, k)^{1/2} & \le \sum_{s\mid c} s^{1/2} \sum_{\substack{k\le K \\ s\mid k}} 1
 \ll 
 \sum_{s\mid c} s^{1/2} \frac{K}{s}\\
 &
 \ll K\sum_{s\mid c} s^{-1/2} 
 \ll K \tau(c) \ll_\eps K c^\eps.
\end{align}
Summing by parts, we conclude that 
\begin{align}A_d\ll_\eps  c^{\eps} 
u^{1/2} (u, hl)^{1/2}
v^{2/3} (v, hl)^{1/3} \delta^{-\eps} + c\delta 
\ll_\eps  
c^{\eps} 
u^{1/2} (u, hl)^{1/2}
v^{2/3} (v, hl)^{1/3},
\end{align}
choosing $\delta = c^{-1/2}$.  Combining this deduction with $B_d \ll \frac{nH}c \sup|\Xi''|$ (and the trivial bound for $b_c$ in the boundary term of \eqref{eq:T_by_parts}), we get
\begin{align}T & 
\ll 
c^{\eps} 
u^{1/2} (u, hl)^{1/2}
v^{2/3} (v, hl)^{1/3}+\sum_{d=1}^{c-1} c^{\eps} 
u^{1/2} (u, hl)^{1/2}
v^{2/3} (v, hl)^{1/3} 
\frac{nH}c \sup|\Xi''| 
\\
&
\ll
(1+\sup|\Xi''| ) nH c^{\eps} 
u^{1/2} (u, hl)^{1/2}
v^{2/3} (v, hl)^{1/3}.
\end{align}
We finally get 
\begin{align}
 |S|^2 & \ll \frac{c^2}H + \frac cH \sum_{1\le h\le H}(1+ \sup|\Xi''|) nH c^{\eps} 
u^{1/2} (u, hl)^{1/2}
v^{2/3} (v, hl)^{1/3}\\
&\ll 
\frac{c^2}H + c^{1+\eps} H n u^{1/2} (u, l)^{1/2}
v^{2/3} (v, l)^{1/3}(1+\sup|\Xi''|),
\end{align}
using 
\eqref{eq:gcd_av} with $h$ in place of $k$. The optimal choice for $H$ is $[c^{1/4}]$, so that 
\begin{align}
 S \ll_\eps (1+ \sup |\Xi''|^{1/2}) c^{5/8+\eps} u^{1/4} v^{1/3} n^{1/2} (u, l)^{1/4}(v, l)^{1/6}.\label{eq:S_l_not_0}
\end{align}
Note that we are not concerned with the value of $\eps$, and thus don't distinguish between $\eps$ and $\eps/2$.
We have thus proved the following proposition. 

\begin{proposition} \label{prop:cancellations}
 Let $S = S_c(k,l, n)$ be the sum defined in \eqref{eq:sum_l}. Write $c = uv$ with $u$ square-free and $v$ square-full. Then, we have
 \begin{align}
  S_c(k,0,n) & \ll_\eps  C n^{1/2} c^{1-1/D+\eps}\\
  S_c(k,l,n) & \ll_\eps (1+ \sup |\Xi''|^{1/2}) c^{5/8+\eps} u^{1/4} v^{1/3} n^{1/2} (u, l)^{1/4}(v, l)^{1/6}.
 \end{align}

\end{proposition}

\section{Fourier decomposition\label{sec:fourier}}

In this section we develop the tools necessary to prove Theorem \ref{th:main} and decompose $f$ into a Fourier series on the torus. We proceed exactly as in   \cite{strombergsson_effective_2013, browning_vinogradov_2016}. 
To begin with we note that 
\begin{align}
 f((1,\bm \xi)M)=f((1,\bm \xi+\bm n)M)
\end{align}
for $\bm n \in \Z^2$. 
So for $M$ fixed, $f$ is a well defined function on $\R^2/\Z^2$ and we can expand it into a Fourier series as 
\beq \label{eq:fourier}f((1,\bm \xi)M)=\sum_{\bm m\in\Z^2}\hat f(M,\bm m) e(\bm m\bm \xi),\eeq
where
\begin{align}\hat f(M,\bm m)=\int_{\T^2}f((1,\bm \xi')M)e(-\bm m \bm \xi')d\bm \xi'.
\end{align}
Note that 
\beq\hat f(TM,\bm m)=\hat f(M,\bm m (T^{-1})^t),
\label{eq:transpose}\eeq
for $T\in\sltz.$ Set $\tilde f_n(M)=\hat f(M,(n,0))$.  
These functions of $M\in\sltr$ are left-invariant under the group $\left(\begin{smallmatrix}1&\Z\\0&1\end{smallmatrix}\right)$ by \eqref{eq:transpose}.

Now it follows from 
 \eqref{eq:transpose}  that
\begin{align}
 \tilde f_n\left(\abcd M\right)=\hat f \left(\abcd M,(n,0)\right)
 &=\hat f \left(M,(n,0)\begin{pmatrix} d & -c\\ -b & a\end{pmatrix}\right)\\
 &=\hat f(M,(nd,-nc)).
\end{align}
Therefore we can rewrite \eqref{eq:fourier} with $\bm \xi= (\xi_1(x),\xi_2(x))$ and $M=\big(\begin{smallmatrix}\sqrt y&x/\sqrt y\\0&1/\sqrt y\end{smallmatrix}\big)$ as 
\beq f\left(\left(1,\bm\xi\right)M\right)=\tilde f_0(M)+\sum_{n\ge 1}\sum_{(c,d)=1} \tilde f_n\left(\begin{pmatrix}*&*\\c&d\end{pmatrix}M\right)e\left(n\left(d\xi_1(x) - c\xi_2(x)\right)\right),\label{eq:afterfourier}\eeq
where $\left(\begin{smallmatrix}*&*\\c&d\end{smallmatrix}\right)=\smallabcd$ is any matrix in \sltz\ with $c$ and $d$ in the second row as specified. 


Integrating  \eqref{eq:afterfourier} over $x$, we obtain
\beq
\int_{\R} f(\tilde n(x)a(y))\,\rho(x)dx=
M(y)+E(y), 
\eeq
where 
\beq\label{eq:main_term}
M(y)=
\int_{\R} \tilde f_0\begin{pmatrix}\sqrt y&x/\sqrt y\\0&1/\sqrt y\end{pmatrix} \,\rho(x)\,dx
\eeq
and 
\beq
E(y)=\sum_{\substack{n\ge1\\ (c,d) =1}} \int_{\R} e\left(n\left(d\xi_1(x) - c\xi_2(x)\right)\right)\tilde f_n\left(\begin{pmatrix}*&*\\c&d\end{pmatrix} \begin{pmatrix}\sqrt y&x/\sqrt y\\0&1/\sqrt y\end{pmatrix}\right)\,\rho(x)dx.\label{eq:errorterms}
\eeq
The main term in this expression is $M(y)$ and, as  is well-known 
(cf.~\cite{flaminio_invariant_2003, strombergsson_deviation_2013}), we have 
\begin{align}M(y)=\int_X f\, d\mu\int_\R \rho(x)dx +O_\eps(\|f\|_{C_{\mathrm{b}}^4} \|\rho\|_{W^{1,1}} y^{1/2-\eps})\end{align}
for every $\eps>0$, where the Sobolev norm of $\rho$ is defined in \eqref{eq:sobolev_norm}.
This statement is nothing more than effective equidistribution of horocycles under the geodesic flow on $\sltz\quot\sltr.$ We need not seek the best error term for this problem, since there will be larger contributions to the error term in Theorem \ref{th:main}.

It remains to estimate $E(y)$  as $y\to 0$, which we do in Section \ref{sec:error}. 


\medskip

We end this section with several technical results that will help us to estimate $E(y)$. First, however, we give a precise definition  of $\|\cdot\|_{C_{\mathrm{b}}^m}$ and $\|\cdot\|_{W^{k,p}}$ for functions on $G$ and hence also on $X$. Following \cite{strombergsson_effective_2013}, we let $\mathfrak g = \Sl(2,\R)\oplus  \R^2$ be the Lie algebra of $G$ and fix  
\begin{align}\begin{aligned}
X_1&=\left(\left(\begin{smallmatrix}0&1\\0&0\end{smallmatrix}\right),\bm 0\right),\quad 
X_2=\left(\left(\begin{smallmatrix}0&0\\1&0\end{smallmatrix}\right),\bm 0\right),\quad
X_3=\left(\left(\begin{smallmatrix}1&0\\0&-1\end{smallmatrix}\right),\bm 0\right),\\
X_4&=\left(\left(\begin{smallmatrix}0&0\\0&0\end{smallmatrix}\right),(1,0)\right),\quad
X_5=\left(\left(\begin{smallmatrix}0&0\\0&0\end{smallmatrix}\right),(0,1)\right)
\end{aligned}\end{align}
to be a basis of $\mathfrak g$. Every element of the universal enveloping algebra $U(\mathfrak g)$ corresponds to a left-invariant differential operator  on functions on $X$.  We define
\beq\label{eq:norm}
\|f\|_{C_{\mathrm{b}}^m} = \sum_{\deg D\le m} \|Df\|_{L^\infty},
\eeq
where the sum runs over monomials in $X_1,\dots,X_5$ of degree at most $m$. 
We also Sobolev norms of functions on \R. For $1\le p< \infty$ and a positive integer $k$, set
\beq\label{eq:sobolev_norm}
\|\rho\|_{W^{k,p}} = \sum_{s=0}^k \|\rho^{(s)}\|_{L^p} = \sum_{s=0}^k \left( \int_\R\nolimits \lvert \rho^{(s)}(x) \rvert^p dx\right)^{1/p}.
\eeq

The following result is
  \cite[Lemma~4.2]{strombergsson_effective_2013}.

\begin{lemma}\label{lem:andreas_fourier1}
Let $m\ge0$ and $n>0$ be integers. Then 
\begin{align}\tilde f_n\abcd \ll_m \frac{\|f\|_{C_{\mathrm{b}}^m}}{n^m(c^2+d^2)^{m/2}}, \quad \forall \abcd \in \SL(2,\R). \end{align}
\end{lemma}

Passing to Iwasawa coordinates in \sltr, we  write 
\beq\label{eq:iwasawa}\tilde f_n(u,v,\theta)=\tilde f_n\left(
 \begin{pmatrix}
  1&u\\
  0&1
 \end{pmatrix}
\begin{pmatrix}
 \sqrt v&0\\
 0&1/\sqrt v
\end{pmatrix}
\begin{pmatrix}
 \cos\theta &-\sin\theta\\
 \sin\theta &\cos\theta
\end{pmatrix}\right).
\eeq
for $u\in \mathbb{R}, v>0$ and $\theta\in \mathbb{R}/2\pi\mathbb{Z}$.
The following  is 
  \cite[Lemma~4.4]{strombergsson_effective_2013}.

\begin{lemma}\label{lem:andreas_fourier2}
Let $m, k_1, k_2, k_3\ge 0$ and $n>0$ be integers, and let $k=k_1+k_2+k_3$. Then
\begin{align}\d_u^{k_1}\d_v^{k_2} \d_\theta^{k_3} \tilde f_n(u,v,\theta)\ll_{m,k} \|f\|_{C_{\mathrm{b}}^{m+k}}n^{-m} v^{m/2-k_1-k_2}.\end{align}
\end{lemma}

As a consequence of Lemma \ref{lem:andreas_fourier1}, we get the bound 
\begin{align}\label{eq:fn_control}\tilde f_n\left(u,\frac{\sin^2 \theta}{c^2 y},\theta\right) \ll_m \|f\|_{C^m_{\text b}} \min\left\{1,\left(\frac{\lvert \sin\theta\rvert }{nc\sqrt y}\right)^m\right\}\end{align} for every integer $m\ge 0$. We also  note that for $a, A, B > 0$ and $B-A> -1$ we have
\beq
\int_{-\pi}^\pi \frac{d\theta}{\lvert \sin \theta\rvert ^A} \min\left\{1,\left(\frac{\lvert \sin \theta\rvert}{a}\right)^B\right\} \ll \min\{a^{-B},a^{-A+1}\}.\label{eq:theta_integral}
\eeq




\section{Error terms\label{sec:error}}
The purpose of this section is to estimate $E(y)$ in 
\eqref{eq:errorterms}.
We begin with the case $c=0$. Then $d=\pm1$ by coprimality, and  \cite[eq.~(25)]{strombergsson_effective_2013} yields
\begin{align}\label{eq:total_c_0}
E_{c= 0}(y)
=
\int_{\R} \tilde f_n\left(\pm \begin{pmatrix}\sqrt y& x/\sqrt y\\0&1/\sqrt y\end{pmatrix}\right)\, \rho(x)dx
\ll
\|f\|_{C_{\mathrm{b}}^2}\frac{y}{n^2}
.
\end{align}
After summing over $n$, the contribution from this  term is clearly much smaller than that claimed in Theorem~\ref{th:main}. 


The remaining contribution ($c\ne 0$) to the error term  $E(y)$ in \eqref{eq:errorterms} is
\beq
E_{c\ne 0}(y)=\sum_{\substack{c\ne0 \\ n\ge1\\ (c,d) =1}}
 \int_{\R} e\left(n\left(d\xi_1(x) - c\xi_2(x)\right)\right)\tilde f_n\left(\begin{pmatrix}*&*\\c&d\end{pmatrix} \begin{pmatrix}\sqrt y&x/\sqrt y\\0&1/\sqrt y\end{pmatrix}\right)\rho(x)dx.
\eeq
Now we proceed to the change of variables,  following \cite[Lemma~6.1]{strombergsson_effective_2013}. Writing the argument of $\tilde f_n$ in Iwasawa coordinates \eqref{eq:iwasawa}, we get 
\begin{align}
\int_\R e\left(n\left(d\xi_1(x) - c\xi_2(x)\right)\right) \tilde f_n \left(\begin{pmatrix}a&b\\c&d\end{pmatrix} \begin{pmatrix}\sqrt y&x/\sqrt y\\0& 1/\sqrt y\end{pmatrix}\right) \rho(x)\, dx = \int_0^\pi g(\theta) d\theta,
\end{align}
for $c>0$, where
\begin{align}
g(\theta) & =  e\left(n\left( d\xi_1\left(-\frac dc+y\ctg \theta\right) -c\xi_2\left(-\frac dc+y\ctg \theta\right)\right)\right)\label{eq:exp_to_simplify}
\\
& \quad \times \tilde f_n\left(\frac{a}c -\frac{\sin 2\theta}{2c^2 y}, \frac{\sin^2\theta}{c^2y},\theta\right)
 \rho\left(-\frac dc + y\ctg\theta\right)
\frac{y}{\sin^2\theta}.\notag
\end{align}
We have the same  integral with limits $-\pi$ and 0 if $c<0$. Combining terms with positive and negative $c$ gives
\begin{align}
E_{c\ne 0}(y)=
\sum_{\substack{ c,n\ge 1}} \sum_{(c,d)=1} \int_{-\pi}^\pi  g(\theta)d\theta.
\end{align}
Let 
\beq
\tilde g(\theta) =  \rho\left(-\tfrac dc \right) \tilde f_n\left(\frac{\bar d}c -\frac{\sin 2\theta}{2c^2y},\frac{\sin^2 \theta}{c^2y},\theta\right) 
e\left ( - n c\Xi\left(-\tfrac dc\right)\right)
\frac{y}{\sin^2\theta},
\eeq
where $\Xi(z)=z\xi_1(z)+\xi_2(z)$. 

\begin{lemma}  
For every $\eps>0$, we have
\begin{equation}\sum_{c,n\ge 1} \sum_{\substack{(c,d)=1\\d \in \Z}} \int_{\theta=-\pi}^\pi  (g(\theta)-\tilde g(\theta))d\theta  \ll_\eps
\|f\|_{C^4_{\text b}} (1+\|\Xi\|_{C^1_{\text b}} + \|\xi_1\|_{L^\infty})\|\rho\|_{W^{1,1}}
y^{1/2 -\eps}
\label{eq:toanalyze}
\end{equation}
for $0<y<1$.
\end{lemma}

\begin{proof}
 We write 
 \begin{align}
  g(\theta) - \tilde g(\theta) & = 
  \tilde f_n\left(\frac{\bar d}c -\frac{\sin 2\theta}{2c^2y},\frac{\sin^2 \theta}{c^2y},\theta\right) 
  \frac{y}{\sin^2\theta}
  \\
  \notag
  &\quad\times
  \big[e\left(n\left( d\xi_1\left(-\tfrac dc+y\ctg \theta\right) -c\xi_2\left(-\tfrac dc+y\ctg \theta\right)\right)\right) \rho\left(-\tfrac dc +y\ctg \theta\right)  \\
  &\quad\quad - e\left (  - n c\Xi\left(-\tfrac dc\right)\right) \rho\left(-\tfrac dc\right)
  \big]
  \notag
  \\
  &
  = 
  \tilde f_n\left(\frac{\bar d}c -\frac{\sin 2\theta}{2c^2y},\frac{\sin^2 \theta}{c^2y},\theta\right)
  \frac{y}{\sin^2\theta}e\left(- nc\Xi \left(-\tfrac dc\right)\right)
  \\
  &
  \quad\times
  \left[e\left(O((\|\Xi\|_{C^1_{\text b}} + \|\xi_1\|_{L^\infty}) ncy \lvert \ctg \theta \rvert\right) \rho\left(-\tfrac dc +y\ctg \theta\right) - \rho\left( -\tfrac dc \right)
  \right].\notag
 \end{align}
 When $ncy \lvert \ctg \theta \rvert < 1,$ we use Taylor expansion of the exponential; in the complementary case we bound it trivially. 
 The contribution of the first option comes in two parts since $e(z) = 1 + O(z)$. 
The first part is controlled by using \eqref{eq:fn_control} with $m=2$ and \eqref{eq:theta_integral} with $A = 2$ and $B=2$, together with elementary inequalities. 
We have the bound
\begin{align}
 \sum_{c,d,n} & \int_{\theta=-\pi}^\pi  \tilde f_n\left(\frac{\bar d}c -\frac{\sin 2\theta}{2c^2y},\frac{\sin^2 \theta}{c^2y},\theta\right)
  \frac{y\, d\theta}{\sin^2\theta}e\left(- nc\Xi \left(-\tfrac dc\right)\right)
  \\\notag
  &\quad\times
  \left[\rho\left(-\tfrac dc +y\ctg \theta\right) - \rho\left( -\tfrac dc \right)
  \right] \1_{ncy\lvert \ctg \theta\rvert <1 }\\
 &\ll \|f\|_{C^2_\text{b}} \int_{-\pi}^\pi\limits \sum_{n,c} \frac{y\, d\theta}{\sin^2 \theta} \min\left\{ 1, \left(\frac{\sin \theta}{nc\sqrt y} \right)^2\right\} 
 \\\notag
 &\quad \times
 \sum_{\mathclap{(c,d) = 1}} \left|\rho\left( -\tfrac dc\right) - \rho\left( - \tfrac dc + y\ctg\theta\right)\right| \1_{ncy\lvert \ctg \theta\rvert <1 }\\
 &\ll  \|f\|_{C^2_\text{b}} \int_{-\pi}^\pi\limits \sum_{n,c} \frac{y\, d\theta}{\sin^2 \theta} \min\left\{ 1, \left(\frac{\sin \theta}{nc\sqrt y} \right)^2\right\} 
 \\
 \notag
 &\quad \times
 \sum_{h\mid c} |\mu(h)| \sum_{d\in \Z} \int_{-\frac{dh}{c}}^{-\frac{dh}{c} + y\ctg \theta} |\rho'(t)|\, dt \, \1_{y\lvert \ctg \theta\rvert <\frac1{nc} }.
\end{align}
Now we use the condition $y\lvert \ctg \theta \rvert <\frac1{nc}$ to recast the sum in $d$ and the integral in $t$ to a single integral over the real line:
\begin{align}
 &
 \ll 
  \|f\|_{C^2_\text{b}} \int_{-\pi}^\pi \sum_{n,c} \frac{y\, d\theta}{\sin^2 \theta} \min\left\{ 1, \left(\frac{\sin \theta}{nc\sqrt y} \right)^2\right\} \sum_{h\mid c} |\mu(h)| \int_{-\infty}^\infty |\rho'(t) | \, dt\\
  &
  \ll_\eps  \|f\|_{C^2_\text{b}} \|\rho\|_{W^{1,1}} \int_{-\pi}^\pi \sum_{n,c} \frac{y\, d\theta}{\sin^2 \theta} \min\left\{ 1, \left(\frac{\sin \theta}{nc\sqrt y} \right)^2\right\} c^\eps 
  \end{align}
At this step we take advantage of \eqref{eq:theta_integral} with $A = 2 = B$, giving the bound
\begin{align}
  &\ll  \|f\|_{C^2_\text{b}} \|\rho\|_{W^{1,1}} \sum_{n,c} yc^\eps \min\left\{(nc\sqrt y)^{-2}, (nc\sqrt y)^{-1}\right\}
  \\
  &
  \ll  \|f\|_{C^2_\text{b}} \|\rho\|_{W^{1,1}} \sum_{k=1}^\infty y k^\eps \min\left\{(k\sqrt y)^{-2}, (k\sqrt y)^{-1}\right\}
  \\
  & \ll \|f\|_{C^2_\text{b}} \|\rho\|_{W^{1,1}} y^{1/2 -\eps}.
\end{align}
The second part is controlled by using \eqref{eq:fn_control} with $m=4$ followed by \eqref{eq:theta_integral} with $A = 3$ and $B=4$, together with elementary inequalities. 
The bound in this case is 
\begin{align}
 \sum_{c,d,n} &\int_{\theta=-\pi}^\pi\tilde f_n \left(\frac{\bar d}c -\frac{\sin 2\theta}{2c^2y},\frac{\sin^2 \theta}{c^2y},\theta\right)
  \frac{y\, d\theta}{\sin^2\theta}e\left(- nc\Xi \left(-\tfrac dc\right)\right)
  \\\notag
  &\quad \times
  \left[ O((\|\Xi\|_{C^1_{\text b}} + \|\xi_1\|_{L^\infty}) ncy \lvert \ctg \theta \rvert   \rho\left(-\tfrac dc +y\ctg \theta\right) 
  \right] \1_{ncy\lvert \ctg \theta\rvert <1 }\\
&\ll \|f\|_{C^4_\text{b}} (\|\Xi\|_{C^1_{\text b}} + \|\xi_1\|_{L^\infty}) \int_{-\pi}^\pi \frac{y\, d\theta}{\sin^2\theta} \sum_{n,c} \min\left\{ 1, \left(\frac{\sin \theta}{nc\sqrt y} \right)^4\right\} ncy \lvert \ctg\theta \rvert 
\\
&
\quad
\times\sum_{h\mid c} |\mu(h)| \sum_{d\in\Z} \rho\left(-\frac{dh}c + y\ctg \theta\right)\1_{ncy\lvert \ctg \theta\rvert <1 }\notag
\\
&\ll \|f\|_{C^4_\text{b}} (\|\Xi\|_{C^1_{\text b}} + \|\xi_1\|_{L^\infty}) \int_{-\pi}^\pi \frac{ncy^2\, d\theta}{\lvert \sin \theta \rvert^3} \sum_{n,c} \min\left\{ 1, \left(\frac{\sin \theta}{nc\sqrt y} \right)^4\right\}  
\\
&
\quad
\times \sum_{h\mid c} |\mu(h)| \left( \frac ch \int_{-\infty}^\infty \rho(t) \, dt + \int_{-\infty}^\infty |\rho'(t)|\,dt \right).\notag
\end{align}
Now the last line is at most a constant times $c^{1+\eps}\|\rho\|_{W^{1,1}}$, since $\rho$ is of integral $1$ and is supported on an interval of length $1$. We get
\begin{align}
 &\ll \|f\|_{C^4_\text{b}} (\|\Xi\|_{C^1_{\text b}} + \|\xi_1\|_{L^\infty})\|\rho\|_{W^{1,1}}
 \sum_{n,c}y^2 n c^{2+\eps} \min\{(nc\sqrt y)^{-2}, (nc\sqrt y)^{-4}\}\\
 &\ll_\eps  \|f\|_{C^4_\text{b}} (\|\Xi\|_{C^1_{\text b}} + \|\xi_1\|_{L^\infty})\|\rho\|_{W^{1,1}}
 \sum_{k=1}^\infty y^2 k^{2+\eps}\min\{(k\sqrt y)^{-2}, (k\sqrt y)^{-4}\}\\
 &\ll_\eps  \|f\|_{C^4_\text{b}} (\|\Xi\|_{C^1_{\text b}} + \|\xi_1\|_{L^\infty})\|\rho\|_{W^{1,1}} y^{1/2 - \eps}.
\end{align}

Now we peruse the second option, $ncy\lvert \ctg \theta \rvert \ge 1$. Again, we distinguish two subcases, $1\ge ncy \ge \lvert \tg \theta\rvert$ and $1 <ncy \ge \lvert \tg \theta\rvert.$ The first subcase is dealt with using \eqref{eq:fn_control} with $m=3$ followed by elementary estimates for the  integral in $\theta$:
\begin{align}
 \sum_{c,d,n} & \int_{\theta=-\pi}^\pi  \left\lvert\tilde f_n\left(\frac{\bar d}c -\frac{\sin 2\theta}{2c^2y},\frac{\sin^2 \theta}{c^2y},\theta\right)\right\rvert
  \frac{y\, d\theta}{\sin^2\theta}
  \\\notag
  &
  \quad \times
  \left[\rho\left(-\tfrac dc +y\ctg \theta\right) + \rho\left( -\tfrac dc \right)
  \right] \1_{1\ge ncy \ge \lvert \tg \theta\rvert}\\
 &\ll
 \|f\|_{C^3_{\text b}} \int_{-\pi}^\pi\limits \sum_{n,c} \frac{y\, d\theta}{\sin^2 \theta} \min\left\{1 , \left(\frac{\lvert \sin \theta\rvert }{nc\sqrt y}\right)^3\right\}
 \\\notag
 &
 \quad \times
 \sum_{\mathclap{(c,d)=1}} \left( \rho\left( - \tfrac dc + y\ctg \theta\right) + \rho\left( -\tfrac dc \right) \right)
 \1_{1\ge ncy \ge \lvert \tg \theta\rvert}\\
 &\ll
 \|f\|_{C^3_{\text b}} \int_{-\pi}^\pi\limits \sum_{n,c} \frac{y\, d\theta}{\sin^2 \theta} \min\left\{1 , \left(\frac{\lvert \sin \theta \rvert}{nc\sqrt y}\right)^3\right\} c^{1+\eps}\|\rho\|_{W^{1,1}} \1_{1\ge ncy \ge \lvert \tg \theta\rvert}
 \end{align}
 Using the same reasoning as before convert the sum over $d$ into an integral, we arrive at the bound
 \begin{align}
 &
 \ll 
 \|f\|_{C^3_{\text b}}\|\rho\|_{W^{1,1}}\sum_{n,c} yc^{1+\eps}\int_{0\le \theta \ll ncy} \frac{\theta \, d\theta}{n^3c^3 y^{3/2}} \1_{1\ge ncy}
 \\
 &
 \ll
 \|f\|_{C^3_{\text b}}\|\rho\|_{W^{1,1}} \sum_{n,c} y^{-1/2} n^{-1} c^{\eps} y^2\1_{1\ge ncy}\\
 &
 \ll 
 \|f\|_{C^3_{\text b}}\|\rho\|_{W^{1,1}} y^{1/2 -\eps}.
\end{align}
In the second subcase, we only keep the condition $ncy > 1$ to get the bound
\begin{align}
\sum_{c,d,n} & \int_{\theta=-\pi}^\pi\limits  \left\lvert\tilde f_n\left(\frac{\bar d}c -\frac{\sin 2\theta}{2c^2y},\frac{\sin^2 \theta}{c^2y},\theta\right)\right\rvert
  \frac{y\, d\theta}{\sin^2\theta}
  \left[\rho\left(-\tfrac dc +y\ctg \theta\right) + \rho\left( -\tfrac dc \right)
  \right] \1_{1< ncy }\\
 &\ll \|f\|_{C^3_{\text b}} \|\rho\|_{W^{1,1}}\sum_{n,c} \int_{-\pi}^\pi \frac{y \, d\theta}{\sin^2\theta} \min\left\{1 , \left(\frac{\lvert \sin \theta\rvert }{nc\sqrt y}\right)^3\right\} c^{1+\eps} \1_{1< ncy }\\
 &
 \ll 
  \|f\|_{C^3_{\text b}}\|\rho\|_{W^{1,1}} \sum_{n,c} \int_{-\pi}^\pi 
  \frac{\lvert \sin\theta \rvert y c^{1+\eps}d\theta}{(nc\sqrt y)^3}\1_{1< ncy }\\
  &
  \ll
  \|f\|_{C^3_{\text b}} \|\rho\|_{W^{1,1}} y^{-1/2}\sum_{k>1/y} \frac{k^{1+\eps}}{k^3}
  \\
  &
  \ll
  \|f\|_{C^3_{\text b}}\|\rho\|_{W^{1,1}} y^{1/2-\eps}.
\end{align}

\end{proof}

\bigskip



We need to analyze 
\begin{align}
 \tilde E_{c\ne 0}(y) = \sum_{c,n\ge 1} \int_{\theta=-\pi}^\pi\limits \sum_{\substack{(c,d) = 1\\ d\in\Z}}
 \tilde f_n\left(\frac{\bar d}c -\frac{\sin 2\theta}{2c^2y},\frac{\sin^2 \theta}{c^2y},\theta\right)
 e\left( - nc\Xi\left(-\frac dc\right)\right) \rho\left(-\frac dc\right)\frac{y \,d\theta}{\sin^2\theta}.
\end{align}

Define Fourier coefficients
\begin{align}
\label{eq:bdef} b_l^{(n,c)}(\theta)& = \int_{0}^1 \tilde f_n\left (u, \frac{\sin^2\theta}{c^2y},\theta\right) e(-lu) \, du,\\
\label{eq:adef} a_k^{(n,c)} & =  \int_0^1 \rho(u) e(-ku) du.
\end{align}
Then, we can use the fact that $\rho$ is supported within $(-1, 0]$ to write 
\begin{align}
 \tilde E_{c\ne 0}(y) & = \sum_{\substack{c,n\ge 1\\ k,l\in\Z}} \int_{\theta=-\pi}^\pi\limits \sum_{\substack{(c,d) = 1\\ 0\le d < c}}
 b_l^{n,c}(\theta)
 a_k^{(n,c)}
 e\left(\frac{l\bar d}c -\frac{l\sin 2\theta}{2c^2y}\right)
 \\
 &
 \notag
 \quad \times
 e\left(-nc\Xi\left(-\frac dc\right)-\frac{kd}c\right)
 \frac{y \,d\theta}{\sin^2\theta}\\
 &\le
 \sum_{\substack{c,n\ge 1\\ k,l\in\Z}} \int_{\theta=-\pi}^\pi 
 |b_l^{n,c}(\theta)
 a_k^{(n,c)}|
 \left\lvert\sum_{\substack{(c,d) = 1\\ 0\le d < c}}
 e\left(\frac{l\bar d - kd}c -
 nc\Xi\left(-\frac dc\right)
 \right)
 \right\rvert
 \frac{y \,d\theta}{\sin^2\theta}. 
\end{align}
Our objective is to get savings for the sum over $d$ and use the bounds 
\begin{align}\label{eq:fourierbound} b_l^{(n,c)}(\theta)\ll \begin{cases}\|f\|_{C_{\mathrm{b}}^m}\min\left\{1, \left(\dfrac{\lvert\sin\theta\rvert}{nc\sqrt y}\right)^m\right\}& \text{for any }m\ge 0,\\
                               l^{-2}\|f\|_{C_{\mathrm{b}}^{m+2}} n^{-4}\min\left\{1, \left(\dfrac{\lvert \sin\theta\rvert}{nc\sqrt y}\right)^{m-4}\right\}& \text{for any }m\ge 4.
                              \end{cases}
\end{align}
and 
\begin{align}
\label{eq:measure_fourier}
a_k\ll_\eta (1+|k|)^{-1-\eta} \|\rho\|_{W^{1,1}}^{1-\eta} \|\rho\|_{W^{2,1}}^\eta,  \text{ for } \eta\in(0,1).
\end{align}
The bound \eqref{eq:fourierbound} is taken from \cite[Lemma~4.1]{browning_vinogradov_2016}, while the bound \eqref{eq:measure_fourier} follows from \eqref{eq:adef} and integration by parts. 
We use the first bound with $m=2$ and $m=6$, and note that 
\begin{align}\label{eq:strominequality}
\int_{-\pi}^\pi \min\left\{1, \left(\frac{\lvert\sin\theta\rvert}a\right)^2\right\} \frac{d\theta}{\sin^2\theta}
\ll  \frac{1}{a(1+a)},
\end{align}
for $a>0$, following \eqref{eq:fn_control}. This inequality will be applied  with $a = nc\sqrt y$. 

We write 
\begin{align}
  S_c(k,l,n) = S = \sum_{\substack{(c,d) = 1\\ 0\le d < c}}
 e\left(\frac{l\bar d - kd}c -
 nc\Xi\left(-\frac dc\right)
 \right).
\end{align}
Cancellations in the exponential sum are proved in Section \ref{sec:number}, where Proposition \ref{prop:cancellations} is established, 
distinguishing two cases, when $l=0$ and when $l\ne 0$. 
Combining contributions of these two cases, we control $\tilde E_{c\ne 0}(y)$ by 
\begin{align}
\tilde E_{c\ne 0}(y) &\ll C \sum_{\substack{c,n\ge 1\\ k\in\Z}} |a_k| \int_{-\pi}^\pi |b_0(\theta)|\frac{y \, d\theta}{\sin^2 \theta} n^{1/2} c^{1-1/D+\eps}  \\ \notag
 &
 \quad + (1+C_2^{1/2}) \sum_{\substack{c, n\ge 1\\ k\in\Z\\ l\ne 0}} |a_k| \int_{-\pi}^\pi |b_l(\theta)|\frac{y \, d\theta}{\sin^2 \theta}  
 c^{\frac58+\eps} u^{\frac14} v^{\frac13} n^{\frac12} (u, l)^{\frac14}(v, l)^{\frac16}
 \\
 & = E_{l=0}(y) + E_{l\ne0}(y).
\end{align}
For $E_{l=0}(y)$ we use the first bound from \eqref{eq:fourierbound} with $m=2$, and \eqref{eq:measure_fourier} followed by \eqref{eq:strominequality}. After bringing out factors of $F =  C\|f\|_{C^2_{\text b}}  \|\rho\|_{W^{1,1}}^{1-\eta} \|\rho\|_{W^{2,1}}^\eta$, we get 
\begin{align}
 E_{l = 0} &
 \ll_{\eta,\eps} F \sum_{c,n,k} (1+|k|)^{-1-\eta} \frac{y}{nc\sqrt y(1+nc\sqrt y)} n^{1/2} c^{1-1/D+\eps} \\
 &
 \ll F \sum_{c, n} \frac{\sqrt y c^{\eps}}{\sqrt n c^{1/D}(1+cn\sqrt y)} \\
 &
 \ll F
 \sum_{n=1}^\infty \sqrt{\frac yn} \left[ \sum_{c\ge \frac{1}{n\sqrt y}} \frac{\1_{n\sqrt y \le  1}}{c^{1+\frac 1D-\eps} n\sqrt y}  + \sum_{c\ge 1} \frac{\1_{n\sqrt y \ge 1}}{c^{1+\frac 1D-\eps} n\sqrt y}  + \sum_{\mathclap{c\le \frac1{n\sqrt y}}} \frac{\1_{n\sqrt y \le 1}}{ c^{\frac1D-\eps}} \right].
\end{align}
Bounding sums over $c$ gives
\begin{align}
 E_{l = 0} &\ll F \bigg[ \sum_{n=1}^\infty \sqrt{\frac yn} \frac{(n\sqrt y)^{1/D-\eps}}{n\sqrt y} \1_{n\sqrt y \le  1}  
 + \sum_{n=1}^\infty \sqrt {\frac yn} \frac1{n\sqrt y}\1_{n\sqrt y \ge 1}
 \\ \notag
 &\quad + \sum_{n=1}^\infty \sqrt{\frac yn} \left(\frac1{n\sqrt y}\right)^{1-\frac 1D +\eps} \1_{n\sqrt y \le 1}  \bigg]
 \\
 &
 \ll F \bigg[
 \sum_{n\le \frac1{\sqrt y}} n^{-\frac32 + \frac 1D -\eps} y^{\frac1{2D} -\eps} + \sum_{n \ge \frac1{\sqrt y}} n^{-3/2} 
 + \sum_{n\le \frac1{\sqrt y}} \frac{y^{\frac1{2D}-\eps}}{n^{\frac32 - \frac1D+\eps}}  \bigg]
 \\
 &
 \ll F
  y^{\frac{1}{2D}-\eps}.
  \label{eq:total_l_0}
\end{align}

When $l\ne 0$, we use \eqref{eq:measure_fourier}, \eqref{eq:fourierbound} with $m=6,$ and \eqref{eq:strominequality}. Abbreviating $H = \|f\|_{C^8_{\text b}}  \allowbreak \|\rho\|_{W^{1,1}}^{1-\eta}  \|\rho\|_{W^{2,1}}^\eta  (1+C_2^{1/2})$, we have
\begin{align}
 E_{l\ne 0} & \ll H \sum_{c,l,n} \frac{l^{-2} n^{-4} y}{nc\sqrt y(1 + nc\sqrt y)} (n^{1/2} c^{5/8+\eps} u^{1/4} v^{1/3} (uv, l)^{1/4})\\
 &\ll
 H \sum_{c,n} \frac{ n^{-4} y}{nc\sqrt y(1 + nc\sqrt y)} (n^{1/2} c^{5/8+\eps} u^{1/4} v^{1/3} )
\end{align}
Now we divide the sum over $c = uv$ into dyadic intervals $[2^{j-1}, 2^j)$, $j\in \N$. This gives 
\begin{align}
 E_{l\ne 0}& \ll H \sum_{n,j} n^{-5} y^{1/2} \sum_{\substack{ v \le 2^j\\ \text{sq.-full}}} \sum_{\substack{2^{j-1}/v \le u\le 2^j/v \\ \text{sq.-free}}} \frac{u^{-1/8+\eps} v^{-1/24+\eps}}{1+nc\sqrt y}\\
 &
 \ll
 H\sum_{n,j} n^{-5} y^{1/2} \sum_{\substack{ v \le 2^j\\ \text{sq.-full}}} \frac{v^{-1/24+\eps}}{1+n2^{j-1}\sqrt y} \sum_{u\le 2^j/v} u^{-1/8 +\eps}\\
 &
 \ll H\sum_{n,j} \frac{n^{-5} y^{1/2} 2^{7j/8+\eps}}{1+n2^{j-1}\sqrt y} \sum_{\substack{ v \le 2^j\\ \text{sq.-full}}} {v^{-1/24+\eps} v^{-7/8+\eps}}.
\end{align}
The sum over $v$ is convergent as square-full numbers have square-root density. The remaining sum gives the bound \begin{align}H y^{1/16-\eps} ,\label{eq:total_l_not_0}\end{align} as needed.

The total error contribution is
\begin{align}
 H y^{1/16-\eps}  + 
 F y^{\frac{1}{2D}-\eps}  + 
 \|f\|_{C^4_{\text b}} y^{1/2-\eps}
 (\|\xi_1\|_{L^\infty} +\|\Xi\|_{C^1_{\text b}}+1)\|\rho\|_{W^{1,1}},
\end{align}
coming from \eqref{eq:total_l_not_0}, \eqref{eq:total_l_0}, and \eqref{eq:toanalyze}, which is majorized by the expression in the statement of the theorem.

\bibliographystyle{plain}
\bibliography{bibliography}

\end{document}